\documentclass[12pt,twoside]{article}
\usepackage{amsmath,amsthm,amssymb,amscd,ascmac, amsfonts}
\usepackage{mathrsfs}
\usepackage{stmaryrd}
\usepackage{braket}
\usepackage{accents}
\usepackage{url}
\allowdisplaybreaks[4]

\numberwithin{equation}{section}
\usepackage[dvips]{graphicx,color,psfrag}

\makeatletter

\newcommand{\tr}{\mathop{\mathrm{tr}}\,}
\renewcommand{\det}{\mathop{\mathrm{det}}\,}

\theoremstyle{plain}
\newtheorem{thm}{Theorem}

\newtheorem{lem}[thm]{Lemma}

\theoremstyle{definition}

\newtheorem{exa}[thm]{Example}

\begin{document}
\title{A revisit to periodic continuants}
\author{Genki Shibukawa}
\date{
\small MSC classes\,:\,05E05, 11B39, 11J70, 33C45}
\pagestyle{plain}

\maketitle


\begin{abstract}
We give a simple proof of some explicit formulas of periodic continuants by Chebyshev polynomials of the second kind given by Rozsa. 
\end{abstract}

\section{Introduction}
Let $\mathbb{Z}_{\geq p}$ be a set of integers greater than or equal to $p$. 
For an integer $p$ and infinite complex sequences $\mathbf{a}_{p}:=(a_{m})_{m \in \mathbb{Z}_{\geq p}}$, $\mathbf{b}_{p}:=(b_{m})_{m \in \mathbb{Z}_{\geq p}}$ and $\mathbf{c}_{p}:=(c_{m})_{m \in \mathbb{Z}_{\geq p}}$, we put 
$$
\boldsymbol{\alpha }_{p}:=(\mathbf{a}_{p}, \mathbf{b}_{p}, \mathbf{c}_{p}).
$$
We define 
the (extended) continuant polynomials $K_{n}(\boldsymbol{\alpha }_{p})$ by  
$$
K_{-1}(\boldsymbol{\alpha }_{p})
   :=
   0, \quad 
K_{0}(\boldsymbol{\alpha }_{p})
   :=
   1, \quad 
K_{1}(\boldsymbol{\alpha }_{p})
   :=
   a_{p}, \quad 
K_{n}(\boldsymbol{\alpha }_{p})
   :=
   \det{T_{n}(\boldsymbol{\alpha }_{p})},
$$
where $T_{n}(\boldsymbol{\alpha }_{p})$ is the following $n\times n$ tridiagonal matrix:
$$
T_{n}(\boldsymbol{\alpha }_{p})
   =
   \begin{pmatrix}
   a_{p} & b_{p} & 0 & \cdots & 0 & 0 \\
   c_{p} & a_{p+1} & b_{p+1} & \cdots & 0 & 0 \\
   0 & c_{p+1} & a_{p+2} & \cdots & 0 & 0 \\
   \vdots & \vdots & \vdots & \ddots & \vdots & \vdots \\
   0 & 0 & 0 & \cdots & a_{p+n-2} & b_{p+n-2} \\
   0 & 0 & 0 & \cdots & c_{p+n-2} & a_{p+n-1} 
   \end{pmatrix}.
$$

There has been many research on continuant polynomials in relation to continued fraction and orthogonal polynomials. 
For continuant polynomials, several properties have been well known since Euler, especially regarding some explicit formulas \cite{T}. 
P. R\'{o}zsa \cite{R} considered for a positive integer $l$ the following $l$-periodicity condition for the sequences $\mathbf{a}_{p}, \mathbf{b}_{p}$ and $\mathbf{c}_{p}$
$$
a_{m+l}=a_{m}, \quad b_{m+l}=b_{m}, \quad c_{m+l}=c_{m}, 
$$
and proposed an explicit formula of $l$-periodic continuant $K_{n}(\boldsymbol{\alpha }_{p})$ by the Chebyshev polynomials of the second kind.

In this article, we give another proof of this explicit formula. 
While R\'{o}zsa's proof is based on direct calculations of the determinant of the definition for continuant polynomials, our proof uses only $2\times 2$ matrices. 
We also mention some examples of this explicit formula related to $q$-continued fractions \cite{MGO}.

\section{Preliminaries}
Throughout the paper, we denote the ring of rational integers by $\mathbb{Z}$. 
We set the Gauss hypergeometric function 
$$
{_{2}F_1}\left(\begin{matrix} a, b \\ c\end{matrix};x\right)
   :=
   \sum_{m\geq 0}\frac{(a)_{m}(b)_{m}}{m!(c)_{m}}x^{m}, \quad 
   (a)_{m}
   :=
   \begin{cases}
   a(a+1)\cdots (a+m-1) & (m\not=0) \\
   1 & (m=0)
   \end{cases}.
$$
Chebyshev polynomial of the second kind is defined by 
\begin{align}
U_{n}(x)
   :=&
   (n+1){_{2}F_1}\left(\begin{matrix} -n, n+2 \\ \frac{3}{2}\end{matrix};\frac{1-x}{2}\right)
   =
      \sum_{k=0}^{n}
   \binom{n}{k}
   \frac{(n+1)_{k+1}}{\left(\frac{3}{2}\right)_{k}}
   \left(\frac{1-x}{2}\right)^{k}, \nonumber \\
\binom{n}{k}
   :=&
   \begin{cases}
   \frac{n(n-1)\cdots (n-k+1)}{k!} & (k\not=0)\\
   1 & (k=0)
   \end{cases}. \nonumber 
\end{align}
It should be remarked that by the definition of $U_{n}(x)$ we have 
$$
U_{0}(x)=1, \quad U_{-1}(x)=0, \quad U_{-2}(x)=-1. 
$$
The generating function for the $U_{n}(x)$ is
\begin{equation}
\label{eq:gen fnc of second Chebyshev}
\frac{1}{1-2xu+u^{2}}
   =
   \sum_{n\geq 0}
   U_{n}(x)u^{n}.
\end{equation}

Let $h_{n}(x,y)$ denote the bivariate complete homogeneous symmetric polynomials of degree $n$ 
$$
h_{n}(x,y)
   :=
   \sum_{i+j=n}x^{i}y^{j}
   =
   \frac{x^{n+1}-y^{n+1}}{x-y}. 
$$
The generating function for the $h_{n}(x,y)$ is
\begin{equation}
\label{eq:gen fnc of comp hom}
\frac{1}{(1-xu)(1-yu)}
   =
   \sum_{n\geq 0}
      h_{n}(x,y)u^{n}.
\end{equation}
By (\ref{eq:gen fnc of second Chebyshev}) and (\ref{eq:gen fnc of comp hom}) we have 
\begin{align}
h_{n}(x,y)
   =
   \begin{cases}
   (xy)^{\frac{n}{2}}U_{n}\left(\frac{x+y}{2\sqrt{xy}}\right) & (xy\not=0) \\
   (x+y)^{n} & (xy=0)
   \end{cases}.
\end{align}

\begin{lem}
Let $A$ be a complex matrix
$$
A
   :=
   \begin{pmatrix}
   a & b \\
   c & d
   \end{pmatrix} 
$$
and $E_{2}$ be the $2\times 2$ identity matrix. 
For any nonnegative integer $m$, we have 
\begin{align}
A^{m}
   &=
   h_{m-1}(\rho _{+},\rho _{-})A-h_{m-2}(\rho _{+},\rho _{-})(\det{A})E_{2} \nonumber \\
\label{eq:power of 2 times 2}
   &=
   \begin{cases}
   (\det{A})^{\frac{m-1}{2}}U_{m-1}\left(\frac{\tr{A}}{2\sqrt{\det{A}}}\right)A
   -(\det{A})^{\frac{m}{2}}U_{m-2}\left(\frac{\tr{A}}{2\sqrt{\det{A}}}\right)E_{2} & (\det{A}\not=0) \\
   (\tr{A})^{m-1}A & (\det{A}=0)
   \end{cases}.\\
\end{align}
Here $\rho _{+}$ and $\rho _{-}$ are the roots of the characteristic polynomial $\det{(\lambda E_{2}-A)}$. 
\end{lem}
\begin{proof}
We consider Euclidean division for $\lambda ^{m}$ and $\det{(\lambda E_{2}-A)}=(\lambda -\rho _{+})(\lambda -\rho _{-})$. 
By the Euclidean theorem for division of polynomials, there exist unique polynomial $q(\lambda )$ and two constants $c_{1},c_{0}$ such that 
\begin{align}
\label{eq:Euclidean division}
\lambda ^{m}
   =
   q(\lambda )(\lambda -\rho _{+})(\lambda -\rho _{-})+c_{1}\lambda +c_{0}.
\end{align}
By substituting $\rho _{\pm }$ for $\lambda $ in (\ref{eq:Euclidean division}), we have
$$
\rho _{\pm }^{m}
   =
   c_{1}\rho _{\pm}+c_{0}.
$$
Hence we obtain 
\begin{align}
c_{1}
   &=
   \frac{\rho _{+}^{m}-\rho _{-}^{m}}{\rho _{+}-\rho _{-}}
   =
   h_{m-1}(\rho _{+},\rho _{-}), \nonumber \\
c_{0}
   &=
   -\frac{\rho _{+}^{m}\rho _{-}-\rho _{+}\rho _{-}^{m}}{\rho _{+}-\rho _{-}}
   =
   -h_{m-2}(\rho _{+},\rho _{-})\det{A}. \nonumber 
\end{align}
We remark that these expressions hold even if the case of $\rho _{+}=\rho _{-}$. 
\end{proof}
\begin{exa}[Power of a quaternion]
Let $a,b,c,d$ be real numbers. 
We define $2\times 2$ matrices $Q$, $I$, $J$ and $K$ by 
\begin{align}
& Q
   =
   Q(a,b,c,d)
   =
   aE_{2}+bI+cJ+dK
   =
   \begin{pmatrix}
   a+b\sqrt{-1} & c+d\sqrt{-1} \\
   -c+d\sqrt{-1} & a-b\sqrt{-1}
   \end{pmatrix}
   \not=
   \begin{pmatrix}
   0 & 0 \\
   0 & 0
   \end{pmatrix},
 \nonumber \\
& I:=
   \begin{pmatrix}
   \sqrt{-1} & 0 \\
   0 & -\sqrt{-1}
   \end{pmatrix}, \quad 
   J:=
   \begin{pmatrix}
   0 & 1 \\
   -1 & 0
   \end{pmatrix}, \quad 
   K:=
   \begin{pmatrix}
   0 & \sqrt{-1} \\
   \sqrt{-1} & 0
   \end{pmatrix}, \nonumber 
\end{align}
which is a matrix realization of a quaternion: 
$$
q=a+bi+cj+dk \in 
\mathbb{H}
   :=
   \{a+bi+cj+dk \mid a,b,c,d \in \mathbb{R}, i^{2}=j^{2}=k^{2}=ijk=-1\}.
$$
By substituting
$$
\det{Q}
   =
   a^{2}+b^{2}+c^{2}+d^{2}
   =:
   |Q|^{2}, \quad 
\tr{Q}
   =
   2a
$$
and $A=Q$ in (\ref{eq:power of 2 times 2}), we have
\begin{align}
Q^{n}
   &=
   |Q|^{n-1}U_{n-1}\left(\frac{a}{|Q|}\right)Q 
   -|Q|^{n}U_{n-2}\left(\frac{a}{|Q|}\right)E_{2} \nonumber \\
   &=
   |Q|^{n}\left(\frac{a}{|Q|}U_{n-1}\left(\frac{a}{|Q|}\right)-U_{n-2}\left(\frac{a}{|Q|}\right)\right)E_{2} \nonumber \\
   & \quad +
   |Q|^{n-1}U_{n-1}\left(\frac{a}{|Q|}\right)(bI+cJ+dK) \nonumber \\
   &=
   \frac{|Q|^{n}}{2}\left(U_{n}\left(\frac{a}{|Q|}\right)-U_{n-2}\left(\frac{a}{|Q|}\right)\right)E_{2} \nonumber \\
   & \quad +
   |Q|^{n-1}U_{n-1}\left(\frac{a}{|Q|}\right)(bI+cJ+dK). \nonumber
\end{align}
The last equality follows from the Pieri formula for $U_{n}(x)$: 
$$
2xU_{n}(x)=U_{n+1}(x)+U_{n-1}(x).
$$
If we put 
$$
Q(a,b,c,d)^{n}
   =
   A_{n}(a,b,c,d)E_{2}+B_{n}(a,b,c,d)I+C_{n}(a,b,c,d)J+D_{n}(a,b,c,d)K,
$$
then we obtain
\begin{align}
A_{n}(a,b,c,d)
   &=
   \frac{|Q|^{n}}{2}\left(U_{n}\left(\frac{a}{|Q|}\right)-U_{n-2}\left(\frac{a}{|Q|}\right)\right), \nonumber \\
B_{n}(a,b,c,d)
   &=
   b|Q|^{n-1}U_{n-1}\left(\frac{a}{|Q|}\right), \nonumber \\
C_{n}(a,b,c,d)
   &=
   c|Q|^{n-1}U_{n-1}\left(\frac{a}{|Q|}\right), \nonumber \\
D_{n}(a,b,c,d)
   &=
   d|Q|^{n-1}U_{n-1}\left(\frac{a}{|Q|}\right). \nonumber
\end{align}
\end{exa}
\begin{lem}[Fundamental properties of continuant polynomials]
\label{thm:Key lemma}
{\rm{(1)}}
\begin{align}
\label{eq:rec of continuant}
K_{-1}(\boldsymbol{\alpha }_{p})
   :=0, \quad 
K_{0}(\boldsymbol{\alpha }_{p})
   :=1, \quad
K_{n}(\boldsymbol{\alpha }_{p})
   =
   a_{p}K_{n-1}(\boldsymbol{\alpha }_{p+1})
   -b_{p}c_{p}K_{n-2}(\boldsymbol{\alpha }_{p+2}). 
\end{align}
{\rm{(2)}} Let 
\begin{align}
L(\alpha ,\beta )
   &:=
   \begin{pmatrix}
   \alpha & \beta \\
   1 & 0
   \end{pmatrix}, \nonumber \\
\label{eq:def of A}
A_{n}(\boldsymbol{\alpha }_{p})
   &:=
   L(a_{p},-b_{p}c_{p})
   L(a_{p+1},-b_{p+1}c_{p+1})
   \cdots 
   L(a_{p+n-1},-b_{p+n-1}c_{p+n-1}).
\end{align}
We have
\begin{align}
\label{eq:exp formula of A}
A_{n}(\boldsymbol{\alpha }_{p})
   =\begin{pmatrix}
K_{n}(\boldsymbol{\alpha }_{p}) & -b_{p+n-1}c_{p+n-1}K_{n-1}(\boldsymbol{\alpha }_{p}) \\
K_{n-1}(\boldsymbol{\alpha }_{p+1}) & -b_{p+n-1}c_{p+n-1}K_{n-2}(\boldsymbol{\alpha }_{p+1})
\end{pmatrix}.
\end{align}
Especially
\begin{align}
\label{eq:tr and det}
\tr{A_{n}(\boldsymbol{\alpha }_{p})}
   =
   K_{n}(\boldsymbol{\alpha }_{p})
   -b_{p+n-1}c_{p+n-1}K_{n-2}(\boldsymbol{\alpha }_{p+1}), \quad 
\det{A_{n}(\boldsymbol{\alpha }_{p})}
   =
   \prod_{j=1}^{n}b_{p+j-1}c_{p+j-1}.
\end{align}
{\rm{(3)}} Put
$$
\mathbf{k}_{n+1}(\boldsymbol{\alpha }_{p})
   :=
   \begin{pmatrix}
   K_{n+1}(\boldsymbol{\alpha }_{p}) \\
   K_{n}(\boldsymbol{\alpha }_{p+1})
   \end{pmatrix}. 
$$
For any integer $m$ such that $n\geq m$, we have 
\begin{align}
\label{eq:n shift op}
\mathbf{k}_{n+1}(\boldsymbol{\alpha }_{p})
   =
   A_{m}(\boldsymbol{\alpha }_{p})\mathbf{k}_{n+1-m}(\boldsymbol{\alpha }_{p+m}).
\end{align}
{\rm{(4)}} If for any integer $n$ 
$
c_{n}=-1,
$
then we have 
\begin{align}
\label{eq:cont frac and continuant}
a_{p} + \underset{i=1}{\overset{n-1}{\mathrm K}} \frac{b_{p+i-1}}{a_{p+i}}
   &=
   \frac{K_{n}(\boldsymbol{\alpha }_{p})}{K_{n-1}(\boldsymbol{\alpha }_{p+1})},
\end{align}
where
$$
a_{p} + \underset{i=1}{\overset{n-1}{\mathrm K}} \frac{b_{p+i-1}}{a_{p+i}}
   :=
   a_{p}+\cfrac{b_{p}}{a_{p+1}+\cfrac{b_{p+1}}{a_{p+2}+\cfrac{b_{p+2}}{\ddots  \cfrac{}{a_{p+n-2}+\frac{b_{p+n-2}}{a_{p+n-1}}}}}}.
$$
\end{lem}
\begin{proof}
{\rm{(1)}} It follows from the definition of $K_{n}(\boldsymbol{\alpha }_{p})$. \\
{\rm{(2)}} When $n=1$, (\ref{eq:exp formula of A}) holds. 
Assume the result true for $n$. 
From induction on $n$ and (\ref{eq:rec of continuant}), we have 
\begin{align}
A_{n+1}(\boldsymbol{\alpha }_{p})
   &=
   L(a_{p},-b_{p}c_{p})
   A_{n}(\boldsymbol{\alpha }_{p+1}) \nonumber \\
   &=
\begin{pmatrix}
a_{p} & -b_{p}c_{p} \\
1 & 0
\end{pmatrix}
\begin{pmatrix}
K_{n}(\boldsymbol{\alpha }_{p+1}) & -b_{p+n}c_{p+n}K_{n-1}(\boldsymbol{\alpha }_{p+1}) \\
K_{n-1}(\boldsymbol{\alpha }_{p+2}) & -b_{p+n}c_{p+n}K_{n-2}(\boldsymbol{\alpha }_{p+2})
\end{pmatrix} \nonumber \\
   &=
\begin{pmatrix}
a_{p}K_{n}(\boldsymbol{\alpha }_{p+1})-b_{p}c_{p}K_{n-1}(\boldsymbol{\alpha }_{p+2}) & -b_{p+n}c_{p+n}(a_{p}K_{n-1}(\boldsymbol{\alpha }_{p+1})-b_{p}c_{p}K_{n-2}(\boldsymbol{\alpha }_{p+2})) \\
K_{n}(\boldsymbol{\alpha }_{p+1}) & -b_{p+n}c_{p+n}K_{n-1}(\boldsymbol{\alpha }_{p+1})
\end{pmatrix} \nonumber \\
   &=
\begin{pmatrix}
K_{n+1}(\boldsymbol{\alpha }_{p+1}) & -b_{p+n}c_{p+n}K_{n}(\boldsymbol{\alpha }_{p+1}) \\
K_{n}(\boldsymbol{\alpha }_{p+1}) & -b_{p+n}c_{p+n}K_{n-1}(\boldsymbol{\alpha }_{p+1})
\end{pmatrix}. \nonumber
\end{align}
{\rm{(3)}} By the definition of $K_{n}(\boldsymbol{\alpha }_{p})$ and (\ref{eq:rec of continuant}), 
\begin{align}
\mathbf{k}_{n+1}(\boldsymbol{\alpha }_{p})
   &=
   \begin{pmatrix}
   K_{n+1}(\boldsymbol{\alpha }_{p}) \\
   K_{n}(\boldsymbol{\alpha }_{p+1})
   \end{pmatrix}
   =
   \begin{pmatrix}
   a_{p}K_{n}(\boldsymbol{\alpha }_{p+1})-b_{p}c_{p}K_{n-1}(\boldsymbol{\alpha }_{p+2}) \\
   K_{n}(\boldsymbol{\alpha }_{p+1})
   \end{pmatrix}
   =
   L(a_{p},-b_{p}c_{p})
   \mathbf{k}_{n}(\boldsymbol{\alpha }_{p+1}). \nonumber 
\end{align}
Hence
\begin{align}
\mathbf{k}_{n+1}(\boldsymbol{\alpha }_{p})
   &=
   L(a_{p},-b_{p}c_{p})
   L(a_{p+1},-b_{p+1}c_{p+1})
   \cdots 
   L(a_{p+m-1},-b_{p+m-1}c_{p+m-1})
   \mathbf{k}_{n+1-m}(\boldsymbol{\alpha }_{p+m}) \nonumber \\
   &=
   A_{m}(\boldsymbol{\alpha }_{p})\mathbf{k}_{n+1-m}(\boldsymbol{\alpha }_{p+m}). \nonumber 
\end{align}
{\rm{(4)}} It follows from (\ref{eq:rec of continuant}) and induction on $n$. 
\end{proof}

\section{Main results}
Under the following we assume 
$$
a_{p+l}=a_{p}, \quad b_{p+l}=b_{p}, \quad c_{p+l}=c_{p} \quad (p \in \mathbb{Z}),
$$
that is to say 
\begin{align}
\label{boldalpha periodicity}
\boldsymbol{\alpha }_{p+l}=\boldsymbol{\alpha }_{p}. 
\end{align}
\begin{thm}
\label{thm:main results}
For any nonnegative integer $m$, we obtain
\begin{align}
\label{eq:main results}
K_{lm}(\boldsymbol{\alpha }_{p})
   &=
   \begin{cases}
   (\det{A_{l}(\boldsymbol{\alpha }_{p})})^{\frac{m-1}{2}}U_{m-1}\left(\frac{\tr{A_{l}(\boldsymbol{\alpha }_{p})}}{2\sqrt{\det{A_{l}(\boldsymbol{\alpha }_{p})}}}\right)K_{l}(\boldsymbol{\alpha }_{p}) & \nonumber \\
   \quad -(\det{A_{l}(\boldsymbol{\alpha }_{p})})^{\frac{m}{2}}U_{m-2}\left(\frac{\tr{A_{l}(\boldsymbol{\alpha }_{p})}}{2\sqrt{\det{A_{l}(\boldsymbol{\alpha }_{p})}}}\right) & (\det{A_{l}(\boldsymbol{\alpha }_{p})}\not=0) \\
   (\tr{A_{l}(\boldsymbol{\alpha }_{p})})^{m-1}K_{l}(\boldsymbol{\alpha }_{p}) & (\det{A_{l}(\boldsymbol{\alpha }_{p})}=0)
   \end{cases}, \\
K_{lm-1}(\boldsymbol{\alpha }_{p+1})
   &=
   \begin{cases}
   (\det{A_{l}(\boldsymbol{\alpha }_{p})})^{\frac{m-1}{2}}U_{m-1}\left(\frac{\tr{A_{l}(\boldsymbol{\alpha }_{p})}}{2\sqrt{\det{A_{l}(\boldsymbol{\alpha }_{p})}}}\right)K_{l-1}(\boldsymbol{\alpha }_{p+1}) & (\det{A_{l}(\boldsymbol{\alpha }_{p})}\not=0) \\
   (\tr{A_{l}(\boldsymbol{\alpha }_{p})})^{m-1}K_{l-1}(\boldsymbol{\alpha }_{p+1}) & (\det{A_{l}(\boldsymbol{\alpha }_{p})}=0)
   \end{cases}.
\end{align}
\end{thm}
\begin{proof}
By (\ref{eq:n shift op}) and periodicity (\ref{boldalpha periodicity})
\begin{align}
\mathbf{k}_{lm}(\boldsymbol{\alpha }_{p})
   &=
   A_{l}(\boldsymbol{\alpha }_{p})\mathbf{k}_{l(m-1)}(\boldsymbol{\alpha }_{p+l})
   =
   A_{l}(\boldsymbol{\alpha }_{p})\mathbf{k}_{l(m-1)}(\boldsymbol{\alpha }_{p}). \nonumber 
\end{align}
Then we have
$$
\mathbf{k}_{lm}(\boldsymbol{\alpha }_{p})
   =
   A_{l}(\boldsymbol{\alpha }_{p})^{m}\mathbf{k}_{0}(\boldsymbol{\alpha }_{p}). 
$$
When 
$$
\det{A_{l}(\boldsymbol{\alpha }_{p})}
   =\prod_{j=1}^{l}b_{p+j-1}c_{p+j-1}\not=0, 
$$
from (\ref{eq:power of 2 times 2}) we have
\begin{align}
A_{l}(\boldsymbol{\alpha }_{p})^{m}
   &=
   (\det{A_{l}(\boldsymbol{\alpha }_{p})})^{\frac{m-1}{2}}U_{m-1}\left(\frac{\tr{A_{l}(\boldsymbol{\alpha }_{p})}}{2\sqrt{\det{A_{l}(\boldsymbol{\alpha }_{p})}}}\right)A_{l}(\boldsymbol{\alpha }_{p}) \nonumber \\
   & \quad -
   (\det{A_{l}(\boldsymbol{\alpha }_{p})})^{\frac{m}{2}}U_{m-2}\left(\frac{\tr{A_{l}(\boldsymbol{\alpha }_{p})}}{2\sqrt{\det{A_{l}(\boldsymbol{\alpha }_{p})}}}\right)E_{2}. \nonumber
\end{align}
If $\det{A_{l}(\boldsymbol{\alpha }_{p})}=0$, then 
$$
A_{l}(\boldsymbol{\alpha }_{p})^{m}
   =
   (\tr{A_{l}(\boldsymbol{\alpha }_{p})})^{m-1}A_{l}(\boldsymbol{\alpha }_{p}). 
$$
Finally, by (\ref{eq:exp formula of A})
\begin{align}
A_{l}(\boldsymbol{\alpha }_{p})
   &=\begin{pmatrix}
K_{l}(\boldsymbol{\alpha }_{p}) & -b_{p+l-1}c_{p+l-1}K_{l-1}(\boldsymbol{\alpha }_{p}) \\
K_{l-1}(\boldsymbol{\alpha }_{p+1}) & -b_{p+l-1}c_{p+l-1}K_{l-2}(\boldsymbol{\alpha }_{p+1})
\end{pmatrix} \nonumber \\
   &=\begin{pmatrix}
K_{l}(\boldsymbol{\alpha }_{p}) & -b_{p-1}c_{p-1}K_{l-1}(\boldsymbol{\alpha }_{p}) \\
K_{l-1}(\boldsymbol{\alpha }_{p+1}) & -b_{p-1}c_{p-1}K_{l-2}(\boldsymbol{\alpha }_{p+1})
\end{pmatrix}. \nonumber 
\end{align}
By comparing the entries of the vector $\mathbf{k}_{lm}(\boldsymbol{\alpha }_{p})$, we obtain the conclusion.  
\end{proof}
Our main result follows from this theorem immediately. 
\begin{thm}
\label{thm:main results cor}
For $j=-1,0,1,\ldots ,l-2$, we have 
\begin{align}
& K_{lm+j}(\boldsymbol{\alpha }_{p-j}) \nonumber \\
   & \quad =
   K_{j}(\boldsymbol{\alpha }_{p-j})K_{lm}(\boldsymbol{\alpha }_{p})-b_{p-1}c_{p-1}K_{j-1}(\boldsymbol{\alpha }_{p-j})K_{lm-1}(\boldsymbol{\alpha }_{p+1}) \nonumber \\
   & \quad =
   \begin{cases}
   (\det{A_{l}(\boldsymbol{\alpha }_{p})})^{\frac{m-1}{2}}U_{m-1}\left(\frac{\tr{A_{l}(\boldsymbol{\alpha }_{p})}}{2\sqrt{\det{A_{l}(\boldsymbol{\alpha }_{p})}}}\right) & \nonumber \\
   \quad \cdot (K_{j}(\boldsymbol{\alpha }_{p-j})K_{l}(\boldsymbol{\alpha }_{p})-b_{p-1}c_{p-1}K_{j-1}(\boldsymbol{\alpha }_{p-j})K_{l-1}(\boldsymbol{\alpha }_{p+1})) & \nonumber \\
   \quad -(\det{A_{l}(\boldsymbol{\alpha }_{p})})^{\frac{m}{2}}U_{m-2}\left(\frac{\tr{A_{l}(\boldsymbol{\alpha }_{p})}}{2\sqrt{\det{A_{l}(\boldsymbol{\alpha }_{p})}}}\right)K_{j}(\boldsymbol{\alpha }_{p-j}) & (\det{A_{l}(\boldsymbol{\alpha }_{p})}\not=0) \\
   (\tr{A_{l}(\boldsymbol{\alpha }_{p})})^{m-1} & \\
   \quad \cdot (K_{j}(\boldsymbol{\alpha }_{p-j})K_{l}(\boldsymbol{\alpha }_{p})-b_{p-1}c_{p-1}K_{j-1}(\boldsymbol{\alpha }_{p-j})K_{l-1}(\boldsymbol{\alpha }_{p+1})) & (\det{A_{l}(\boldsymbol{\alpha }_{p})}=0)
   \end{cases}.\\
\label{eq:main results 2}
\end{align}
Here we define $K_{-2}(\boldsymbol{\alpha }_{p+1})$ by 
$$
-b_{p-1}c_{p-1}K_{-2}(\boldsymbol{\alpha }_{p+1})
   :=
   K_{0}(\boldsymbol{\alpha }_{p-1})
   =1. 
$$ 
\end{thm}
\begin{proof}
From (\ref{eq:n shift op}) and (\ref{eq:exp formula of A}), we have
\begin{align}
\mathbf{k}_{lm+j}(\boldsymbol{\alpha }_{p-j})
   &=
   A_{j}(\boldsymbol{\alpha }_{p-j})
   \mathbf{k}_{lm}(\boldsymbol{\alpha }_{p}) \nonumber \\
\label{eq:the above}
   &=
   \begin{pmatrix}
   K_{j}(\boldsymbol{\alpha }_{p-j}) & -b_{p-1}c_{p-1}K_{j-1}(\boldsymbol{\alpha }_{p-j}) \\
   K_{j-1}(\boldsymbol{\alpha }_{p-j+1}) & -b_{p-1}c_{p-1}K_{j-2}(\boldsymbol{\alpha }_{p-j+1})
   \end{pmatrix}
   \begin{pmatrix}
   K_{lm}(\boldsymbol{\alpha }_{p}) \\
   K_{lm-1}(\boldsymbol{\alpha }_{p+1})
   \end{pmatrix}. 
\end{align}
By (\ref{eq:main results}) and comparing the entires of (\ref{eq:the above}), we obtain our main result (\ref{eq:main results 2}). 
\end{proof}

\section{Examples}
In this section, we give the examples of (\ref{eq:main results}) for $l=1,2,3$ explicitly. 

\subsection{$l=1$}
In this subsection, we put
$$
a:=a_{p}=a_{p+1}, \quad 
b:=b_{p}=b_{p+1}, \quad 
c:=c_{p}=c_{p+1}.
$$
In the case of $l=1$, since  
$$
\tr{A_{1}(\boldsymbol{\alpha }_{p})}
   =a, \quad 
\det{A_{1}(\boldsymbol{\alpha }_{p})}
   =bc, 
$$
we have the following well-known result:
\begin{align}
\label{eq:l1}
K_{m}(\boldsymbol{\alpha }_{p})
   =
   \begin{cases}
   (bc)^{\frac{m-1}{2}}U_{m-1}\left(\frac{a}{2\sqrt{bc}}\right) & (bc\not=0) \\
   a^{m-1} & (bc=0)
   \end{cases}.
\end{align}

\subsection{$l=2$}
In this subsection, we put
\begin{align}
& a_{1}:=a_{2m+1}=a_{2m+3}, \quad 
b_{1}:=b_{2m+1}=b_{2m+3}, \quad 
c_{1}:=c_{2m+1}=c_{2m+3}, \nonumber \\
& a_{2}:=a_{2m}=a_{2m+2}, \quad 
b_{2}:=b_{2m}=b_{2m+2}, \quad 
c_{2}:=c_{2m}=c_{2m+2}. \nonumber
\end{align}
Since 
$$
A_{2}(\boldsymbol{\alpha }_{p})
   =
   \begin{pmatrix}
   a_{1}a_{2}-b_{p}c_{p} & -a_{p}b_{p+1}c_{p+1} \\
   a_{p+1} & -b_{p+1}c_{p+1}
   \end{pmatrix}
$$
and
\begin{align}
\tr{A_{2}(\boldsymbol{\alpha }_{p})}
   &=
   a_{1}a_{2}-b_{1}c_{1}-b_{2}c_{2}, \quad 
\det{A_{1}(\boldsymbol{\alpha }_{p})}
   =b_{1}c_{1}b_{2}c_{2}, \nonumber 
\end{align}
our main result (\ref{eq:main results 2}) is  
\begin{align}
\label{eq:l2 example}
K_{2m}(\boldsymbol{\alpha }_{p})
   &=
   \begin{cases}
   (b_{1}c_{1}b_{2}c_{2})^{\frac{m-1}{2}}U_{m-1}\left(\frac{a_{1}a_{2}-b_{1}c_{1}-b_{2}c_{2}}{2\sqrt{b_{1}c_{1}b_{2}c_{2}}}\right)(a_{1}a_{2}-b_{p}c_{p}) & \nonumber \\
   \quad -(b_{1}c_{1}b_{2}c_{2})^{\frac{m}{2}}U_{m-2}\left(\frac{a_{1}a_{2}-b_{1}c_{1}-b_{2}c_{2}}{2\sqrt{b_{1}c_{1}b_{2}c_{2}}}\right) & (b_{1}c_{1}b_{2}c_{2}\not=0) \\
   (a_{1}a_{2}-b_{1}c_{1}-b_{2}c_{2})^{m-1}(a_{1}a_{2}-b_{p}c_{p}) & (b_{1}c_{1}b_{2}c_{2}=0)
   \end{cases}, \\
K_{2m-1}(\boldsymbol{\alpha }_{p+1})
   &=
   \begin{cases}
   (b_{1}c_{1}b_{2}c_{2})^{\frac{m-1}{2}}U_{m-1}\left(\frac{a_{1}a_{2}-b_{1}c_{1}-b_{2}c_{2}}{2\sqrt{b_{1}c_{1}b_{2}c_{2}}}\right)a_{p+1} & (b_{1}c_{1}b_{2}c_{2}\not=0) \\
   (a_{1}a_{2}-b_{1}c_{1}-b_{2}c_{2})^{m-1}a_{p+1} & (b_{1}c_{1}b_{2}c_{2}=0)
   \end{cases}.
\end{align}
\begin{exa}[A $q$-analogue of Fibonacci numbers]
Morier-Genoud and Ovsienko \cite{MGO} introduced the following notion of $q$-deformed rational numbers and continued fractions, motivated by Jones polynomials of rational knots or $F$-polynomials of a cluster algebra. 
For a positive rational number $\frac{r}{s}$ and its (regular) continued fraction
$$
\frac{r}{s}
   =
   a_{1} + \underset{i=1}{\overset{2n-1}{\mathrm K}} \frac{1}{a_{i+1}}, \quad a_{1},\ldots ,a_{2n}>0,
$$
their $q$-analogue are defined by 
$$
\left[\frac{r}{s}\right]_{q}
   :=
   [a_{1}]_{q} + \underset{i=1}{\overset{2n-1}{\mathrm K}} \frac{q^{(-1)^{i-1}a_{i}}}{[a_{i+1}]_{q^{(-1)^{i}}}},
$$
where $q\not=0$ is a complex parameter and 
$$
[a]_{q}:=\frac{1-q^{a}}{1-q}.
$$
Our formulas (\ref{eq:main results 2}) or (\ref{eq:l2 example}) are useful to write down these $q$-analogue explicitly. 
We consider the case of 
\begin{equation}
\label{eq:l2 specialization}
a_{p}=[a]_{q^{(-1)^{p-1}}}, \quad 
b_{p}=q^{(-1)^{p-1}a}, \quad  
c_{p}=-1,
\end{equation}
where $a$ is a positive integer. 
By substituting (\ref{eq:l2 specialization}) in (\ref{eq:l2 example}), we have
\begin{align}
K_{2m}(\boldsymbol{\alpha }_{p})
   &=
   U_{m-1}\left(\frac{[a]_{q}[a]_{q^{-1}}+q^{a}+q^{-a}}{2}\right)([a]_{q}[a]_{q^{-1}}+q^{(-1)^{p-1}a}) \nonumber \\
   & \quad -U_{m-2}\left(\frac{[a]_{q}[a]_{q^{-1}}+q^{a}+q^{-a}}{2}\right), \nonumber \\
K_{2m-1}(\boldsymbol{\alpha }_{p+1})
   &=
   U_{m-1}\left(\frac{[a]_{q}[a]_{q^{-1}}+q^{a}+q^{-a}}{2}\right)[a]_{q^{(-1)^{p}}}.  
\end{align}
From (\ref{eq:cont frac and continuant}), we derive an explicit formula 
\begin{align}
& [a]_{q} + \underset{i=1}{\overset{2n-1}{\mathrm K}} \frac{q^{(-1)^{i-1}a}}{[a]_{q^{(-1)^{i}}}} \nonumber \\
   &=
   \frac{U_{m-1}\left(\frac{[a]_{q}[a]_{q^{-1}}+q^{a}+q^{-a}}{2}\right)([a]_{q}[a]_{q^{-1}}+q^{(-1)^{p-1}a})-U_{m-2}\left(\frac{[a]_{q}[a]_{q^{-1}}+q^{a}+q^{-a}}{2}\right)}{U_{m-1}\left(\frac{[a]_{q}[a]_{q^{-1}}+q^{a}+q^{-a}}{2}\right)[a]_{q^{(-1)^{p}}}}. 
\end{align}

In particular, The case of $a=1$ is a $q$-analogue of Fibonacci numbers defined by 
$$
F_{1}=1, \quad F_{2}=1, \quad F_{n+2}=F_{n+1}+F_{n}.
$$
As is well known, the continued fraction of $\frac{F_{2n+1}}{F_{2n}}$ is given by 
$$
\frac{F_{2n+1}}{F_{2n}}
   =
      1 + \underset{i=1}{\overset{2n-1}{\mathrm K}} \frac{1}{1}.
$$
Thus, a $q$-analogue of this rational number and continued fraction expansion are equal to 
$$
\left[\frac{F_{2n+1}}{F_{2n}}\right]_{q}
   =
      1 + \underset{i=1}{\overset{2n-1}{\mathrm K}} \frac{q^{(-1)^{i-1}}}{1}.
$$
We put
\begin{align}
& a_{1}=a_{2}=1, \quad 
b_{i}=q^{(-1)^{i-1}}, \quad 
c_{1}=c_{2}=-1 \nonumber
\end{align}
and
$$
F_{2m+2}(q)
   :=
   K_{2m+1}(\boldsymbol{\alpha }_{2}), \quad 
F_{2m+1}(q)
   :=
   K_{2m}(\boldsymbol{\alpha }_{1}).
$$
From Lemma \ref{thm:Key lemma} {\rm{(4)}}, we have
$$
\left[\frac{F_{2n+1}}{F_{2n}}\right]_{q}
   =
   \frac{F_{2n+1}(q)}{F_{2n}(q)}.
$$
This sequence $\{F_{n}(q)\}$ satisfies 
$$
F_{1}(q)=1, \quad F_{2}(q)=1, \quad F_{3}(q)=1+q 
$$
and 
$$
F_{2m}(q)=F_{2m-1}(q)+q^{-1}F_{2m-2}(q), \quad 
F_{2m+1}(q)=F_{2m}(q)+qF_{2m-1}(q).
$$
From (\ref{eq:l2 example}), $\{F_{n}(q)\}$ has the following explicit expression: 
\begin{align}
F_{2m+2}(q)
   &=
   U_{m-1}\left(\frac{1+q+q^{-1}}{2}\right)
   (1+q^{-1})
   -U_{m-2}\left(\frac{1+q+q^{-1}}{2}\right) \nonumber \\
F_{2m+2}(q)
   &=
   U_{m-1}\left(\frac{1+q+q^{-1}}{2}\right). 
\end{align}
\end{exa}

\subsection{$l=3$}
Put 
\begin{align}
& a_{1}:=a_{3m+1}=a_{3m+4}, \quad 
b_{1}:=b_{3m+1}=b_{3m+4}, \quad 
c_{1}:=c_{3m+1}=c_{3m+4}, \nonumber \\
& a_{2}:=a_{3m+2}=a_{3m+5}, \quad 
b_{2}:=b_{3m+2}=b_{3m+5}, \quad 
c_{2}:=c_{3m+2}=c_{3m+5}, \nonumber \\
& a_{3}:=a_{3m}=a_{3m+3}, \quad 
b_{3}:=b_{3m}=b_{3m+3}, \quad 
c_{3}:=c_{3m}=c_{3m+3}. \nonumber
\end{align}
From (\ref{eq:exp formula of A}) and (\ref{eq:tr and det}), we have
$$
A_{3}(\boldsymbol{\alpha }_{p})
   =
   \begin{pmatrix}
   a_{1}a_{2}a_{3}-a_{p+2}b_{p}c_{p}-a_{p}b_{p+1}c_{p+1} & -a_{p}a_{p+1}b_{p+2}c_{p+2}+b_{p}c_{p}b_{p+2}c_{p+2} \\
   a_{p+1}a_{p+2}-b_{p+1}c_{p+1} & -a_{p+1}b_{p+2}c_{p+2}
   \end{pmatrix}
$$
and
\begin{align}
\tr{A_{3}(\boldsymbol{\alpha }_{p})}
   &=
   a_{1}a_{2}a_{3}-a_{1}b_{2}c_{2}-a_{2}b_{3}c_{3}-a_{3}b_{1}c_{1}, \quad 
\det{A_{1}(\boldsymbol{\alpha }_{p})}
   =\prod_{j=1}^{3}b_{j}c_{j}. \nonumber 
\end{align}
Then (\ref{eq:main results 2}) can be written as 
\begin{align}
& K_{3m+1}(\boldsymbol{\alpha }_{p-1}) \nonumber \\
   & \quad =
   a_{p-1}K_{3m}(\boldsymbol{\alpha }_{p})-b_{p-1}c_{p-1}K_{3m-1}(\boldsymbol{\alpha }_{p+1}) \nonumber \\
   & \quad =
   \begin{cases}
   \prod_{j=1}^{3}(b_{j}c_{j})^{\frac{m-1}{2}}U_{m-1}\left(\frac{a_{1}a_{2}a_{3}-a_{1}b_{2}c_{2}-a_{2}b_{3}c_{3}-a_{3}b_{1}c_{1}}{2\sqrt{b_{1}c_{1}b_{2}c_{2}b_{3}c_{3}}}\right) & \nonumber \\
   \quad \cdot 
(a_{p-1}(a_{1}a_{2}a_{3}-a_{p+2}b_{p}c_{p}-a_{p}b_{p+1}c_{p+1}) & \nonumber \\
   \quad \quad -b_{p-1}c_{p-1}(a_{p+1}a_{p+2}-b_{p+1}c_{p+1})) & \nonumber \\
   \quad -\prod_{j=1}^{3}(b_{j}c_{j})^{\frac{m}{2}}U_{m-2}\left(\frac{a_{1}a_{2}a_{3}-a_{1}b_{2}c_{2}-a_{2}b_{3}c_{3}-a_{3}b_{1}c_{1}}{2\sqrt{b_{1}c_{1}b_{2}c_{2}b_{3}c_{3}}}\right)a_{p-1} & (\prod_{j=1}^{3}b_{j}c_{j}\not=0) \\
   (a_{1}a_{2}a_{3}-a_{1}b_{2}c_{2}-a_{2}b_{3}c_{3}-a_{3}b_{1}c_{1})^{m-1} & \nonumber \\
   \quad \cdot 
(a_{p-1}(a_{1}a_{2}a_{3}-a_{p+2}b_{p}c_{p}-a_{p}b_{p+1}c_{p+1}) & \nonumber \\
   \quad \quad -b_{p-1}c_{p-1}(a_{p+1}a_{p+2}-b_{p+1}c_{p+1}))
   & (\prod_{j=1}^{3}b_{j}c_{j}=0)
   \end{cases}, \nonumber \\
& K_{3m}(\boldsymbol{\alpha }_{p}) \nonumber \\
   & \quad =
   \begin{cases}
   \prod_{j=1}^{3}(b_{j}c_{j})^{\frac{m-1}{2}}U_{m-1}\left(\frac{a_{1}a_{2}a_{3}-a_{1}b_{2}c_{2}-a_{2}b_{3}c_{3}-a_{3}b_{1}c_{1}}{2\sqrt{b_{1}c_{1}b_{2}c_{2}b_{3}c_{3}}}\right) & \nonumber \\
   \quad \cdot 
   (a_{1}a_{2}a_{3}-a_{p+2}b_{p}c_{p}-a_{p}b_{p+1}c_{p+1}) & \nonumber \\
   \quad -\prod_{j=1}^{3}(b_{j}c_{j})^{\frac{m}{2}}U_{m-2}\left(\frac{a_{1}a_{2}a_{3}-a_{1}b_{2}c_{2}-a_{2}b_{3}c_{3}-a_{3}b_{1}c_{1}}{2\sqrt{b_{1}c_{1}b_{2}c_{2}b_{3}c_{3}}}\right) & (\prod_{j=1}^{3}b_{j}c_{j}\not=0) \\
   (a_{1}a_{2}a_{3}-a_{1}b_{2}c_{2}-a_{2}b_{3}c_{3}-a_{3}b_{1}c_{1})^{m-1} & \nonumber \\
   \quad \cdot 
   (a_{1}a_{2}a_{3}-a_{p+2}b_{p}c_{p}-a_{p}b_{p+1}c_{p+1}) & (\prod_{j=1}^{3}b_{j}c_{j}=0)
   \end{cases}, \\
& K_{3m-1}(\boldsymbol{\alpha }_{p+1}) \nonumber \\
   & \quad =
   \begin{cases}
   \prod_{j=1}^{3}(b_{j}c_{j})^{\frac{m-1}{2}}U_{m-1}\left(\frac{a_{1}a_{2}a_{3}-a_{1}b_{2}c_{2}-a_{2}b_{3}c_{3}-a_{3}b_{1}c_{1}}{2\sqrt{b_{1}c_{1}b_{2}c_{2}b_{3}c_{3}}}\right) & \nonumber \\
   \quad \cdot (a_{p+1}a_{p+2}-b_{p+1}c_{p+1}) & (\prod_{j=1}^{3}b_{j}c_{j}\not=0) \\
   (a_{1}a_{2}a_{3}-a_{1}b_{2}c_{2}-a_{2}b_{3}c_{3}-a_{3}b_{1}c_{1})^{m-1} \nonumber \\
   \quad \cdot (a_{p+1}a_{p+2}-b_{p+1}c_{p+1}) & (\prod_{j=1}^{3}b_{j}c_{j}=0)
   \end{cases}. \\
\end{align}

\bibliographystyle{amsplain}

\noindent 
Department of Mathematics, Graduate School of Science, Kobe University, \\
1-1, Rokkodai, Nada-ku, Kobe, 657-8501, JAPAN\\
E-mail: g-shibukawa@math.kobe-u.ac.jp
\end{document}